\documentclass[12pt,reqno]{amsart}
\advance \topmargin by -\headheight
\advance \topmargin by -\headsep
\evensidemargin \oddsidemargin
\usepackage{amsmath}
\usepackage{amsfonts}
\usepackage{stmaryrd}
\usepackage{amssymb}
\usepackage{amsthm}
\usepackage{csquotes}
\usepackage{color}
\usepackage{mathrsfs}
\usepackage{dsfont}
\usepackage{bbm}
\usepackage{cite}
\usepackage{soul}

\numberwithin{equation}{section}

\newtheorem{theorem}{Theorem}[section]
\newtheorem{corollary}[theorem]{Corollary}
\newtheorem{lemma}[theorem]{Lemma}

\usepackage{mathtools}
\DeclarePairedDelimiter{\ceil}{\lceil}{\rceil}

\def\RR{\mathbb{R}}

\def\e{\mathbbm{1}}

\def\d{\,\mathrm{d}}

\def\TT{\mathbb{T}}

\def\ZZ{\mathbb{Z}}

\title[Kemperman's inequality and Freiman's lemma via few translates]{Kemperman's inequality and Freiman's lemma via few translates}

\author{Yifan Jing}
\address{Mathematical Institute, University of Oxford, Oxford OX2 6GG, UK}
\email{jing@maths.ox.ac.uk}

\author{Akshat Mudgal}
\address{Mathematical Institute, University of Oxford, Oxford OX2 6GG, UK}
\email{mudgal@maths.ox.ac.uk}

\thanks{YJ and AM are supported by Ben Green’s Simons Investigator Grant, ID:376201. }


\begin{document}

\begin{abstract}

    Let $G$ be a connected compact group equipped with the normalised Haar measure $\mu$. Our first result shows that given $\alpha, \beta>0$, there is a constant $c = c(\alpha,\beta)>0$ such that for any compact sets $A,B\subseteq G$ with 
    \[ \alpha\mu(B)\geq\mu(A)\geq  \mu(B) \ \ \text{and} \ \  \mu(A)+\mu(B)\leq 1-\beta,  \]
    there exist $b_1,\dots b_c\in B$ such that
    \[
    \mu(A\cdot \{b_1,\dots,b_c\})\geq  \mu(A)+\mu(B).
    \]
    A special case of this, that is, when $G=\mathbb{T}^d$, confirms a recent conjecture of Bollob\'as, Leader and Tiba. 
    
    We also prove a quantitatively stronger version of such a result in the discrete setting of $\mathbb{R}^d$. Thus, given $d \in \mathbb{N}$, we show that there exists $c = c(d) >0$ such that for any finite, non-empty set $A \subseteq \mathbb{R}^d$ which is not contained in a translate of a hyperplane, one can find $a_1, \dots, a_c \in A$ satisfying
    \[ |A+ \{a_1, \dots, a_c\}| \geq (d+1)|A| - O_d(1). \]
    The main term here is optimal and recovers the bounds given by Freiman's lemma up to the $O_d(1)$ error term.
\end{abstract}
\maketitle

\section{Introduction}

Given finite, non-empty sets $A, B$ of integers, a classical result in additive number theory implies that the sumset 
\[ A+ B = \{ a + b : a \in A, b \in B\} \]
satisfies
\begin{equation} \label{fir}
|A+B| \geq |A| + |B| - 1. 
\end{equation} 
This has been since generalised to a variety of settings, including the case when $A,B$ are subsets of some connected, compact group $G$ as well as when $A,B$ are finite subsets of $\mathbb{R}^d$. In order to elucidate upon the former, we describe some notation, and so, whenever we write $G$ to be a compact, connected group, we will associate with $G$ the group operation $\cdot$ and the normalised Haar measure $\mu$ satisfying $\mu(G) = 1$. 
In this setting, a fundamental result of Kemperman~\cite{Kemperman} implies that whenever $A,B \subseteq G$ are non-empty compact sets, then the product set
\[ A \cdot B = \{ a \cdot b : a \in A, b \in B\} \]
satisfies
\begin{equation} \label{kemp}
 \mu(A \cdot B) \geq \min\{\mu(A)+\mu(B), 1 \}.
 \end{equation}
When $G$ is abelian, this is also known as Kneser's inequality~\cite{Kneser}. Moving now from the continuous case to the discrete setting, we note that analysis of finite sumsets in $\mathbb{R}^d$ arose from work of Freiman \cite{Fr1973} and has played an important role in additive combinatorics, specially due to its connections to Freiman's theorem and the sum-product problem, see, for instance, \cite{Ch2003, GT2006}. In particular, denoting $\dim(A)$ to be the dimension of the affine span of $A$ for any finite, non-empty $A \subseteq \RR^d$, the well-known Freiman's lemma (see \cite[Section 1.14]{Fr1973} or \cite[Lemma 5.13]{TV2006}) implies that whenever $\dim(A) = d$, then 
\begin{equation} \label{fl}
    |A+A| \geq (d+1)|A| - d(d+1)/2. 
\end{equation} 
Both the aforementioned inequalities can be seen to be sharp, for example, one can observe that \eqref{kemp} is optimal by considering the case when $G = \mathbb{T}$ is the torus and $A,B$ are sufficiently small intervals in $\mathbb{T}$. Similarly, \eqref{fl} is sharp as evinced by the case when $A \subseteq \mathbb{R}^d$ satisfies
\begin{equation} \label{chr}
A = \{0, e_1, \dots, e_{d-1}\} \times \{1, 2,\dots,N\},
\end{equation} 
where $\{e_1, \dots, e_d\}$ form the canonical basis of $\RR^d$.


In a recent paper, Bollob\'as, Leader, and Tiba~\cite{BLT22} considered a different perspective towards \eqref{fir}, wherein, they showed that whenever $A, B \subseteq \mathbb{Z}$ are finite sets with $|A| \geq |B| \geq 1$, then there exist $b_1, b_2, b_3 \in B$ such that
\begin{equation} \label{jkl}
    |A+ \{b_1, b_2, b_3\}| \geq |A| + |B| - 1 . 
\end{equation} 
They proved some similar results when $A, B$ are subsets of $\mathbb{F}_p$ and $\mathbb{T}$, with various restrictions on the sizes of $A,B$. They further conjectured that an analogous phenomenon should hold in the higher dimensional torus $\TT^d$, that is, writing $G$ to be $\mathbb{T}^d$, there exists some $c>0$ such that for any compact, non-empty sets $A, B \subseteq G$ with $\mu(A) = \mu(B) = 1/3$, one can find $b_1, \dots, b_c \in B$ such that
\[ \mu(A \cdot \{b_1, \dots, b_c\}) \geq \mu(A) + \mu(B). \]
Here, since $G = \mathbb{T}^d$, the group operation $\cdot$ denotes the additive group operation in $\mathbb{T}^d$.

Our first result confirms a much more general version of their conjecture.

\begin{theorem}\label{thm:main}
Given $\alpha, \beta>0$, there exists some constant $c= c(\alpha, \beta)>0$ such that the following holds. Let $G$ be a connected compact group and let $A,B\subseteq G$ be compact sets with
\[ \alpha\mu(B)\geq\mu(A)\geq  \mu(B) \ \ \text{and} \ \  \mu(A)+\mu(B)\leq 1-\beta.  \]
Then there exist $b_1,\dots,b_c\in B$ such that \[
\mu(A\cdot\{b_1,\dots,b_c\})\geq \mu(A)+\mu(B). 
\]
\end{theorem}

Hence, we prove their conjecture for all connected, compact groups, including the cases when $G$ is possibly nonabelian. Moreover, our constant $c$ is independent of $G$. One way to interpret our result is that given suitable sets $A, B \subseteq G$, one can recover the lower bound \eqref{kemp} for $\mu(A \cdot B)$ by just considering few translates $A \cdot b_1, \dots, A \cdot b_c$ of $A$. For instance, a conclusion akin to that of Theorem \ref{thm:main} holds for any $K$-approximate group $A$ of $G$, that is, symmetric sets $A \subseteq G$ containing the identity for which $A \cdot A \subseteq X \cdot A$ for some $X \subseteq G$ with $|X| \leq K$. Indeed in our setting, we may apply \eqref{kemp} to deduce that
\[ 2\mu(A) \leq \mu(A \cdot A) \leq \mu(A \cdot \{x_1, \dots, x_c\} ), \]
where $\{x_1, \dots, x_c\} = X^{-1} \subseteq G$. On the other hand, a celebrated result of Breuillard, Green, and Tao \cite{BGT} shows that approximate groups are highly structured objects that essentially only come from nilpotent groups.
Despite the rarity of such objects, Theorem \ref{thm:main} suggests that one can in fact recover the above inequality for any compact set $A \subseteq G$ with $\mu(A) \leq 1- \beta$ and $c = c(\beta)$.

In the discrete setting, it is worth noting that the main result of \cite{BLT22} already implies \eqref{jkl} for finite, non-empty sets $A, B \subseteq \RR^d$ by considering Freiman isomorphisms from $A \cup B$ to sets of integers. On the other hand, noting Freiman's lemma, it is natural to ask whether one can also recover bounds akin to \eqref{fl} by just considering $O_d(1)$ translates of $A$. This is precisely the content of our next result. 

\begin{theorem} \label{dis}
Given $d \in \mathbb{N}$, there exists some constant $c= c(d)>0$ such that for every finite, non-empty set $A \subseteq \RR^d$ with $\dim(A) = d$, there exist $a_1, \dots, a_{c} \in A$ satisfying
\[ |A+ \{a_1, \dots, a_{c}\}| \geq (d+1)|A| - 5(d+1)^3.\]
\end{theorem}

The main term in the above lower bound matches the main term in \eqref{fl} provided by Freiman's lemma, and as before, can be seen to be sharp by considering the sets represented in \eqref{chr}. A similar setup has been analysed in the recent work of Fox, Luo, Pham and Zhou \cite{FLPZ22}, wherein, the authors showed that for every $d,\varepsilon >0$, there exist constants $C = C(d,\varepsilon) \geq 1$ and $c = c(d,\varepsilon) \geq 1$ such that any finite set $A \subseteq \mathbb{R}^d$, which is not contained in $C$ translates of any hyperplane, contains elements $a_1,\dots, a_c$ satisfying
\[ |A + \{a_1,\dots, a_c\}| \geq (2^d - \varepsilon) |A|. \]
Thus, while they obtain a stronger lower bound for $|A+\{a_1,\dots, a_c\}|/|A|$, they must restrict to relatively less structured sets $A$, that is, sets $A$ which can not be covered by $C$ translates of any hyperplane. Moreover, the quantitative dependence of $C$ on $d,\varepsilon$ is slightly weak, since this arises from an application of the Freiman--Bilu theorem (see \cite{GT2006}). In comparison, Theorem \ref{dis} provides lower bounds for $|A+\{a_1,\dots, a_c\}|/|A|$ that grow linearly with $d$ but hold for any $d$-dimensional subset $A\subseteq \RR^d$, essentially allowing us to obtain optimal estimates when $C = 1$.

We now turn to a related philosophy of considering inverse results for small sumsets. This forms a large collection of results in additive combinatorics, with the central theorem being Freiman's inverse theorem, which suggests that sets $A \subseteq \mathbb{Z}$ with $|A+A| \leq K|A|$ must be contained efficiently in very additively structured sets known as generalised arithmetic progressions. Moreover, obtaining quantitatively optimal versions of this result as $K$ becomes large is a key open problem in the area, see \cite{Sa2013}. On the other hand, when $K$ is very small, say $K < 3$, then we have very sharp characterisations. This is the content of Freiman's $3k-4$ theorem (see \cite[Theorem 5.11]{TV2006}) which implies that any finite set $A \subseteq \mathbb{Z}$ with $|A+A| \leq 3|A|-4$ is contained in an arithmetic progression $P$ of length $|A+A| - |A|+1$. Such types of results have been extended to a variety of settings including the case of finite fields and $\ZZ^d$, see \cite{CSS2020, St2010} and the references therein.

More recently, it was shown in \cite{BLT22b} that one can obtain an analogue of Freiman's $3k-4$ theorem for sumsets akin to \eqref{jkl}, that is, instead of assuming upper bounds for $|A+A|$ with $A\subseteq \mathbb{Z}$, one can assume that $|A+A'| \leq (2 + \varepsilon)|A|$ for every $A' \subseteq A$ with $|A'| =4$ and with $\varepsilon$ being sufficiently small, whereupon, one may deduce that $A$ is contained in an arithmetic progression $P$ of size $(1+ \varepsilon + O(\varepsilon^2))|A|$. In our paper, we are also able to prove these type of inverse theorems in the settings of Theorems \ref{thm:main} and \ref{dis}, and we present the first of these below.

\begin{theorem}\label{thm:main inverse}
Given $\varepsilon,\alpha, \beta>0$, there exist constants $c,\eta$  depending only on $\varepsilon,\alpha, \beta$, such that the following holds. 
Let $G$ be a connected compact group, let $A,B\subseteq G$ be compact sets such that $\mu(A)+\mu(B)<1-\beta$, and $\alpha^{-1}\mu(B)\leq\mu(A)\leq \alpha\mu(B)$. Moreover, suppose that for every $b_1,\dots,b_c\in B$, we have
\[
\mu(A\cdot\{b_1,\dots,b_c\})\leq \mu(A)+\mu(B)+\eta \min\{\mu(A),\mu(B)\}. 
\]
Then there is a surjective group homomorphism $\chi:G\to \TT$  and two compact intervals $I,J\subseteq \TT$, with $\lambda$ being the normalised Lebesgue measure on $\TT$, such that
\[
\lambda(I)\leq (1+\varepsilon)\mu(A) \ \ \text{and} \ \  \lambda(J)\leq (1+\varepsilon)\mu(B) \ \ \text{and} \ \  A\subseteq \chi^{-1}(I) \ \ \text{and} \ \ B\subseteq \chi^{-1}(J). \]
\end{theorem}

Here, if one replaces the assumption that $\mu(A\cdot\{b_1,\dots,b_c\})$ is small from Theorem~\ref{thm:main inverse} by the hypothesis that the entire product set $\mu(A\cdot B)$ is small, then the corresponding inverse theorem was obtained by the first author and Tran in~\cite{JT21}, and when $G$ is abelian, it was first proven by Tao~\cite{Tao18}, see also \cite{ChristIliopoulou} for a quantitatively better bound. This asserts that sets with doubling close to $2$ in connected compact groups are dominated by a one-dimensional torus, and, in particular, when $G$ is compact semisimple, the doubling constant of any small subset should be away from $2$. Theorem~\ref{thm:main inverse} is a strengthening of this phenomenon, and we have the following immediate corollary.

\begin{corollary} \label{corollary}
There are absolute constants $c,\eta>0$ such that the following holds. 
    Let $G$ be a compact semisimple Lie group, and let $A\subseteq G$ satisfy $\mu(A)\leq 1/3$. Then there exist $a_1,\dots,a_c \in A$ such that 
    \[
   \mu(A\cdot\{a_1,\dots,a_c\})>(2+\eta)\mu(A). 
    \]
\end{corollary}

Returning to the discrete setting, Stanchescu \cite{St2010} showed that whenever a large set $A \subseteq \ZZ^d$ with $\dim(A) =d$ has its sumset close in size to the lower bound in \eqref{fl}, then $A$ is contained in a union of $d$ parallel lines, that is, $A \subseteq T + l = \cup_{t \in T} (t+l),$ where $T \subseteq \ZZ^d$ is a set satisfying $|T| \leq d$ and $l$ is some one dimensional subspace, consequently making progress towards a problem raised by Freiman \cite{Fr1999}. An asymmetric version of the above conclusion for sets $A, B \subseteq \RR^d$ may be derived by combining ideas from \cite{Mu2019} and \cite{Mu2022}, see, for instance, the proof of Lemma \ref{cov}. Both these results seem to capture the extremality of the example presented in \eqref{chr}. Using our methods, we are able to obtain an analogous conclusion under a weaker hypothesis, that is, instead of assuming upper bounds for the entire sumset $A+A$, we operate under the assumption that for any $A' \subseteq A$ with $|A'| \ll_d 1$, the sumset $A + A'$ is close in size to the estimates provided by Theorem \ref{dis}. We record this inverse result below.

\begin{theorem} \label{mnz}
Given $d \in \mathbb{N}$, there exists some constant $c= c(d)>0$ such that the following holds true. Let $A \subseteq \RR^d$ be a finite, non-empty set with $\dim(A) = d$, such that for any $a_1, \dots, a_{c} \in A$, we have
\[ |A+ \{a_1, \dots, a_{c}\}| \leq (d+1 + 1/16)|A|.\]
Then either $|A| \ll_{d} 1$ or $A \subseteq T + l$, where $l$ is a one dimensional subspace of $\mathbb{R}^d$ and $T \subseteq \mathbb{R}^d$ satisfies $|T| \leq (d+1)^2$.
\end{theorem}

We point out that we have not chosen to optimise the constant $1/16$ in the above result, and this can be quantitatively improved, potentially at the cost of slightly increasing the upper bound for $|T|$. It would also be interesting to show a variant of the above result where one obtains $|T| \leq d$, akin to the results in \cite{Mu2022, St2010} as well as Lemma \ref{cov} in \S5.

We now provide a brief outline of the paper. We utilise \S2 to record various preliminary definitions and results that we will require throughout our paper. These include a variety of inverse and structural results from additive combinatorics along with some standard lemmata like the Pl\"{u}nnecke-Ruzsa inequality. We employ \S3 to present the proofs of Theorems \ref{thm:main} and \ref{thm:main inverse}. The key lemma in the section is Lemma~\ref{lem: A B when A is large}, which deal with the case when $G$ has a torus quotient and the sets are close to $1$-dimensional Bohr sets. The proofs of Theorems \ref{thm:main} and \ref{thm:main inverse} is that either we are in the situation to apply Lemma~\ref{lem: A B when A is large}, or a random selection argument works. 
Our main aim of \S4 will be to prove Lemma \ref{lin}, which can be interpreted as a variant of Theorem \ref{dis} in the case when our set $A \subseteq \RR^d$ can be covered by few translates of a line. This will require a combination of combinatorial geometric and additive combinatorial methods. Moreover, Lemma \ref{lin} will then naturally combine with Theorem \ref{mnz} to deliver the proof of Theorem \ref{dis}. Finally, in \S5, we record the proof of Theorem  \ref{mnz}.

\textbf{Notation.} In this paper, we use Vinogradov notation, that is, we write $X \gg_{z} Y$, or equivalently $Y \ll_{z} X$, to mean $|X| \geq C_{z} |Y|$ where $C$ is some positive constant depending on the parameter $z$. Moreover, we write $X = O_{z}(Y)$ to mean $X \ll_z Y$. Given a group $G$ and a set $A \subseteq G$, we use $\e_A$ to denote the indicator function of the set $A$, that is, $\e_A(g) = 1$ when $g \in A$, and $\e_A(g) = 0$ when $g \in G \setminus A$.

\textbf{Acknowledgements.} We would like to thank Marcelo Campos and Zach Hunter for useful comments. We are also grateful to Zach for help in improving the presentation of our paper.

\section{Preliminaries}

A  locally compact group $G$ is a group equipped with a locally compact and Hausdorff topology on its underlying set such that the group multiplication and inversion maps are continuous. We say that a measure $\mu$ on the collection of Borel subsets of $G$ is a {\em left Haar measure} if it satisfies the following properties: 
\begin{enumerate}
     \item (nonzero) $\mu(X)>0$ for all open $X\subseteq G$;
     \item (left-invariant) $\mu(X) =\mu(a\cdot X)$ for all $a\in G$ and all measurable sets $X \subseteq G$;
    \item (inner regular) when $X$ is open, $\mu(X) =\sup \mu(K)$ with $K$ ranging over compact subsets of $X$;
    \item (outer regular) when $X$ is Borel, $\mu(X) =\inf \mu(U)$ and $U$ ranging over open subsets of $G$ containing $X$;
    \item (compactly finite) $\mu$ takes finite measure on compact subsets of $G$.
\end{enumerate}
When $G$ is a locally compact topological group, a famous theorem of Haar asserts that $G$ has a unique (up to a constant factor) left-invariant Haar measure, denoted by $\mu$. When $G$ is compact, $\mu$ is also right-invariant (that is for a measureable $X$ we also have $\mu(X \cdot g)=\mu(X)$ for all $g\in G$). We say $\mu$ is {\em normalised} if $\mu(G)=1$. 

In our proof of Theorem \ref{thm:main}, it will sometimes be more convenient to study the {\em popular product set} defined as follows. 
Given $0\leq t\leq\min\{\mu(A),\mu(B)\}$, we denote the popular product set 
\[
A\cdot_{t}B=\{x\in G \mid \e_A*\e_B(x)\geq t\},
\]
where, given any $f,g : G \to \mathbb{R}$, we define the convolution function $f*g : G \to \mathbb{R}$ as 
\[ f*g(x)=\int_G f(y)g(y^{-1}x)\d \mu(y) \]
for every $x \in G$. Note that when $G$ is compact, we also have 
\[ f*g(x)=\int_G f(xy^{-1})g(y)\d \mu(y) \]
for every $x \in G$, since $\mu$ is bi-invariant. Noting these definitions, one can further see that
\[
\lim_{t\to 0}A\cdot_t B = \mathrm{supp} (\e_A*\e_B) \subseteq A \cdot B . 
\]
We now state the well-known quotient integral formula, see, for example, \cite[Theorem 2.49]{Folland}. 
\begin{lemma}\label{lem: quotient integral formula}
   Let $G$ be a locally compact group and $H\leq G$ be a closed normal subgroup. Let $\mu_G$ and $\mu_H$ be left invariant Haar measures on $G$ and $H$ respectively. Then there is a unique left invariant Haar measure $\mu_{G/H}$ on $G/H$ such that for every compactly supported continuous function $f:G\to\mathbb C$,
   \[
   \int_G f(g)\d\mu_G (g) =\int_{G/H}\int_H f(xh) \d\mu_H (h)\d\mu_{G/H} (xH). 
   \]
\end{lemma}

A locally compact group $G$ is called {\em connected} if it is connected as a topological space. In particular, it does not have any open subgroups. Recall that open subgroups are closed, and closed subgroups of finite indices are open, hence all the closed subgroups of a connected group have measure $0$.

We now recall Kemperman's inequality as presented in \eqref{kemp}, which can be seen as an extension of Kneser's inequality \cite{Kneser}, which in turn, is a generalisation of the Cauchy--Davenport inequality. We remark that the latter can further generalised to all locally compact groups see~\cite{JT-cauchy}. Moreover, one can prove various inverse theorems for when the lower bound in such inequalities is close to being sharp. In our proof of Theorem \ref{thm:main}, we will make use of one such inverse theorem which can be derived from the work of the first author and Tran in ~\cite{JT21}; a generalised version of which is also stated in the forthcoming work \cite{J23}. The abelian case of this theorem (with a weaker quantitative bound, in particular, with $\eta=o(1)$) was first obtained by Tao~\cite{Tao18}, with a further result by Christ--Iliopoulou~\cite{ChristIliopoulou} which dispensed a sharp exponent bound.

\begin{theorem} \label{thm: inverse Kemperman}
Given $\varepsilon,\alpha>0$, there is $\eta>0$ such that the following hold. 
Let $G$ be a connected compact group equipped with the normalised Haar measure $\mu$ and let $A,B\subseteq G$ be compact sets such that $\alpha \mu(B) \geq \mu(A)\geq \mu(B)$. If 
\[
\mu(A\cdot_{\eta \mu(B)} B)\leq \mu(A)+\mu(B)+\eta \mu(B)<1,
\]
then there is a continuous surjective homomorphism $\chi: G\to\TT$, and two compact intervals $I,J\subseteq \TT$ such that with $\lambda$ the normalised Lebesgue measure on $\TT$ we have
\[
\lambda(I)\leq (1+\varepsilon)\mu(A),\qquad \lambda(J)\leq (1+\varepsilon)\mu(B),
\]
and $A\subseteq \chi^{-1}(I)$, $B\subseteq \chi^{-1}(J)$. 
\end{theorem}

We now move to the discrete setting, where the ambient group is $\RR^d$ and $A \subseteq \RR^d$ is a finite set. Our first preliminary lemma in this setting will be the following inverse theorem for sumsets in higher dimensions from \cite{Mu2019}.

 \begin{lemma} \label{m1}
 Given $d \in \mathbb{N}$ and $K \geq 1$ and a finite set $A \subseteq \mathbb{R}^d$ with $|A+A| \leq K|A|$, there exist parallel lines $l_1, \dots, l_r$ in $\mathbb{R}^d$ and a constant $0 < \sigma \leq 1/2$ depending only on $K$, such that
 \[ |A \cap l_1| \geq \dots | A \cap l_r| \geq |A \cap l_1|^{1/2} \gg |A|^{\sigma}   \]
 and
 \[ |A \setminus (l_1 \cup \dots \cup l_r) | \ll K |A|^{1 - \sigma}.   \]
 \end{lemma}

Thus if $A \subseteq \RR^d$ has a small sumset, then one can efficiently cover $A$ with translates of a one dimensional subspace. We will combine this with the following asymmetric variant of Freiman's lemma from the second author's work in \cite{Mu2022}.

\begin{lemma} \label{m2}
    Let $d \geq 2$ be an integer, let $A, B \subseteq \mathbb{R}^d$ be finite sets such that $|A| \geq |B|$ and $\dim(A) = d$. Suppose that $l_1, \dots, l_r,m_1, \dots, m_q$ are parallel lines such that 
    \[ A \subseteq l_1 \cup \dots \cup l_r \ \ \text{and}  \ \ B \subseteq m_1 \cup \dots \cup m_q, \]
    with $|A \cap l_i|, |B \cap m_j| \geq 1$ for every $1 \leq i \leq r$ and $1 \leq j \leq q.$    Then we have that
    \[ |A+B| \geq |A| + \Big( d + 1 - \frac{1}{r-d + 2} - \frac{1}{q - c+ 2}\Big)|B| - (d-1)(r+q),\]
    where $c =d$ when $\dim(B) = d$ and $c = \dim(B)$ when $\dim(B) < d$.
\end{lemma}


Next, we will require the following modified version of \cite[Theorem 8']{BLT22} for our proof of Theorem \ref{mnz}.

\begin{lemma} \label{btb}
For all $K$ and $\varepsilon>0$, there exists $c = c(K,\varepsilon)>0$ such that the following holds true. Given a finite subset $A$ of some abelian group $G$, there exists a set $A^* \subseteq A$ with $|A^*| \geq (1- \varepsilon)|A|$ such that if we select $a_1, \dots, a_c$ uniformly at random from $A^*$ then 
\[ \mathbb{E}_{a_1, \dots, a_c \in A^*} |A^* + \{a_1, \dots, a_c\}| \geq \min\{(1- \varepsilon) |A^*+A^*|, K|A^*|\}. \]
\end{lemma}

This may be obtained by applying the following corollary of \cite[Theorem 1.1]{sh2019} in the proof of \cite[Theorem 8']{BLT22}, instead of the original version of \cite[Theorem 1.1]{sh2019}. 

\begin{lemma} \label{shm}
Given $\varepsilon>0$ and $K \geq1$, the following is true for all $\delta>0$ sufficiently small in terms of $\varepsilon,K$. Let $A$ be a finite subset of some abelian group $G$ and let $\Gamma \subseteq A \times A$, with $|\Gamma| \geq (1- \delta)|A|^2$. Writing $S = \{a+b : (a,b) \in \Gamma\}$, suppose that $|S|\leq K|A|$.
Then there exists $A'' \subseteq A$ such that
\[ |A''| \geq (1-\varepsilon) |A| \ \ \text{and} \ \ |A''+A''| \leq |S| + \varepsilon |A|. \]
\end{lemma}

Lemma \ref{shm} follows from a straightforward manner from \cite[Theorem 1.1]{sh2019} by setting $A'' = A' \cap B'$ in the conclusion of \cite[Theorem 1.1]{sh2019} and rescaling $\varepsilon$ appropriately.

We now record a standard result in additive combinatorics known as the Pl\"{u}nnecke-Ruzsa inequality \cite[Corollary 6.29]{TV2006}, which, in the situation when $|A+B| \leq K|B|$, allows us to efficiently bound many-fold sumsets of $A$.

\begin{lemma} \label{prin}
    Let $A,B$ be finite, non-empty subsets of some abelian group $G$ satisfying $|A+B|\leq K|B|$. Then for every $k \in \mathbb{N}$, we have that
    \[  | \{a_1  + \dots + a_k : a_1, \dots a_k \in A \}| \leq K^{k} |B|. \]
\end{lemma}

We end this section by recalling that \eqref{jkl}, when combined with the theory of Freiman isomorphisms, implies that for any finite $A, B \subseteq \RR^d$ with $|A| \geq |B|,$ one may find elements $b_1,b_2,b_3\in B$ such that
\begin{equation} \label{b12}
    |A + \{b_1,b_2,b_3\} | \geq |A| + |B| - 1. 
\end{equation}

\section{Proof of Theorems \ref{thm:main} and \ref{thm:main inverse}}

We begin this section by proving the following lemma that can be seen to provide a stronger version of the conclusion from Theorem \ref{thm:main}, but in the more specific setting where $A,B \subseteq G$ can be mapped surjectively to large subsets of some intervals in $\TT$.



\begin{lemma}\label{lem: A B when A is large}
Let $G$ be a connected compact group with $\mu$ being the normalised Haar measure on $G$, and let $A,B$ be two compact sets in $G$ with $\mu(A)\geq \mu(B)$. Suppose  that $\chi:G\to\TT$ is a surjective compact group homomorphism and $I,J$ are two compact intervals in $\TT$ with 
\[ \lambda(I)+\lambda(J)<1, \  \text{and} \ A\subseteq \chi^{-1}(I) \ \text{and} \ B\subseteq \chi^{-1}(J), \]
where $\lambda$ is the normalised Lebesgue measure on $\TT$. Then there exist $b_1,b_2,b_3\in B$ such that 
\[ \mu(A \cdot \{b_1,b_2,b_3\})\geq \mu(A)+\mu(B).  \]
\end{lemma}
\begin{proof}
    By replacing $A$ and $B$ with $a^{-1} \cdot A$ and $B \cdot b^{-1}$ for some $a\in A$ and $b\in B$ respectively, we may assume that $0$ is the left end point for both $I$ and $J$. By the compactness assumption on $A$ and $B$ we may further assume that
    \[
    \chi^{-1}(0)\cap A\neq\varnothing,\quad\chi^{-1}(0)\cap B\neq\varnothing,\quad\chi^{-1}(\lambda(I))\cap A\neq\varnothing,\quad\chi^{-1}(\lambda(J))\cap B\neq\varnothing.
    \]
Note that as $G$ is compact and hence unimodular, such translations would not affect the measure of the product set $A \cdot  B$. 

Let us now consider the following short exact sequence
\[
1\to H:= \ker\chi \to G \xrightarrow{\chi} \TT \to 1,
\]
and note that $H$ is a connected compact group, write $\mu_H$ to denote the normalised Haar measure on $H$. As $\lambda(I)+\lambda(J)<1$, observe that the natural embedding $\phi:\TT\to\RR$ with $\phi(\TT)=[0,1)$ preserves the measure of the sumset 
\[ I+J = \{ i + j : i \in I, j \in J \}.\]
For clarity of exposition, we note that here, and throughout the proof of Lemma \ref{lem: A B when A is large}, we will use $+$ to denote the additive operation in the abelian group $\mathbb{T}$. Next, defining $\psi=\phi\circ\chi$, which maps $G$ to $\RR$, we will abuse some notation and write $I$, $J$ (instead of $\phi(I)$,$\phi(J)$) for the compact intervals in $\RR$. Note that we still have
\[
    \psi^{-1}(0)\cap A\neq\varnothing,\quad\psi^{-1}(0)\cap B\neq\varnothing,\quad\psi^{-1}(\lambda(I))\cap A\neq\varnothing,\quad\psi^{-1}(\lambda(J))\cap B\neq\varnothing ,
    \]
where we now view $\lambda$ as its pushforward in $\RR$. 
 Let 
 \[ b_1\in \psi^{-1}(0)\cap B \ \ \text{and} \ \  b_2\in \psi^{-1}(\lambda(J))\cap B,\]
 and let
\[
X= A \cdot \{b_1,b_2\} =  (A \cdot b_1) \cup (A \cdot b_2). 
\]
We may assume that $\lambda(I)\geq\lambda(J)$, because otherwise we would have 
\[ \mu(X)\geq 2\mu(A)\geq \mu(A)+\mu(B), \]
in which case, we would be done. For every element $x$ in $\RR$ 
and for every set $S \subseteq G$,
let us consider the fiber function with respect to $A$ defined as
\[
f_A(x)=\mu_H( ([\psi^{-1}(x)]^{-1} \cdot A) \cap H) \ \ \ \text{for every} \ x \in \psi(A), 
\]
where $[xH]$ denotes the representative element that lives on the coset $xH$. Moreover, we write $f_A(x) = 0$ for every $x \in \mathbb{R} \setminus \psi(A)$. We define the fiber functions for $B$ and $X$ in the same way. 
We further set
\[ \pi(X):= \int_{0}^{\lambda(J)} \max_{y\in x+\lambda(J)\ZZ} f_X(y)  \d\lambda. \]
This can be viewed as the size of some sort of a ``maximum projection'' of $X$ onto $[0,\lambda(J))$. It is worth pointing out that in the above definition while we are considering, for every $x \in [0, \lambda(J))$, the maximum over all $y \in x+\lambda(J)\ZZ$, in practice, this maximum only considers finitely many values of $y$ since the function $f_X$ is only supported on the set $[0, \lambda(J) + \lambda(I)]$.

Our first aim is to show  that 
\begin{equation} \label{xs}
    \mu(X)\geq \mu(A)+\pi(X).
\end{equation} 
We proceed by observing that
\[ f_X(x) = f_A(x) \ \text{for} \ x \in [0,\lambda(J)) \ \ \text{and} \ \ f_X(x) = f_A(x - \lambda(J))  \ \text{for all}  \ x > \lambda(I),\]
and so, we may apply Lemma~\ref{lem: quotient integral formula} to discern that
\[
\mu(X)=\int_{\lambda(I)-\lambda(J)}^{\lambda(I)}f_A(x)\d\lambda(x)+\int_{0}^{\lambda(I)}f_X(x)\d\lambda(x),
\]
which, in turn, gives us
\[
\mu(X) - \mu(A) =  \int_{0}^{\lambda(I)} (f_X(x)-f_A(x-\lambda(J)) )\d\lambda(x).
\]
Note that in the above expression, $f_A(x - \lambda(J)) =0$ for every $x \in [0, \lambda(J))$. Splitting the integral from $[0, \lambda(I)]$ over periods of length $\lambda(J)$, we may now deduce that
\begin{equation} \label{ex3}
\mu(X) - \mu(A) = \int_0^{\lambda(J)}\sum_{k=0}^\infty (f_X(x+k\lambda(J))-f_A(x+(k-1)\lambda(J)) ) \d\lambda(x).
\end{equation}
While for expository purposes we are taking the sum above over all $k \in \mathbb{N} \cup \{0\}$, we remark that as before, it is, in practice, a finite sum since for all $k > 2 \ceil{ \lambda(I)/ \lambda(J)}$ we have 
\[ f_X(x + k \lambda(J)) = f_A(x + (k-1) \lambda(J)) = 0 \]
as $X,A \subseteq [0, \lambda(I) + \lambda(J)]$.
Now, fixing $x_0 \in [0, \lambda(J))$, we let $k_0 = k_0(x_0)\in \mathbb{N} \cup \{0\}$ satisfy
\[ f_X(x_0 + k_0\lambda(J)) = \max_{y \in x + \lambda(J) \mathbb{Z}} f_X(y) . \]
Since
\[ f_X(x)\geq \max\{f_A(x),f_A(x-\lambda(J))\}\]
for every $x \in \mathbb{R}$ and $f_A(x - \lambda(J)) = 0$ for every $x < \lambda(J)$, we see that
\begin{align*}
    \sum_{k=0}^\infty  (f_X(x_0+k\lambda(J)) & -f_A(x_0+(k-1)\lambda(J)) )  
     \geq  \sum_{k = 0}^{k_0} (f_X(x_0+k\lambda(J))-f_A(x_0+(k-1)\lambda(J)) ) \\
 &  = f_X(x_0 + k_0 \lambda(J)) + \sum_{k=0}^{k_0-1} ((f_X(x_0+k\lambda(J))-f_A(x_0+k\lambda(J)) \\
 & \geq f_X(x_0 + k_0 \lambda(J)).
    \end{align*}
Integrating the above for all $x_0 \in [0, \lambda(J))$ and noting \eqref{ex3}, we obtain the claimed estimate
\[ \mu(X) - \mu(A) \geq \int_{0}^{\lambda(J)}\max_{y \in x + \lambda(J) \mathbb{Z}} f_X(y) = \pi(X). \]

In the rest of the proof we will assume that $\pi(X)<\mu(B)$ as otherwise we are done by applying \eqref{xs}. Note that for every $a\in A$, we have
\begin{align*}
    \mu((a \cdot B)\setminus X) 
    &\geq \max\{\mu(a \cdot B)-\mu(X\cap \psi^{-1}([\psi(a),\psi(a)+\lambda(J)])),0\} \\
    & \geq \max\{\mu(B)-\pi(X),0\}.
\end{align*}
By the assumption that $\pi(X)<\mu(B)$, we have, for every $a\in A$, the inequality
\[
\mu((a \cdot B ) \setminus X)\geq \mu(B)-\pi(X). 
\]
Let us now choose $b_3\in B$ uniformly at random with respect to $\mu$. Then by Fubini's theorem and the above inequality, we find that
\begin{align*}
\mathbb{E}(\mu( (A \cdot b_3) \setminus X))& = \frac{1}{\mu(B)}\int_G \e_A(a) \int_G \e_B(b) \e_{G\setminus X}(a \cdot b) \d\mu(b)\d\mu(a)\\
&=\frac{1}{\mu(B)}\int_G\e_A(a)\mu((a\cdot B)\setminus X) \d\mu(a) \\
& \geq \frac{\mu(A)(\mu(B)-\pi(X))}{\mu(B)}. 
\end{align*}
Therefore there exists $b_3\in B$ such that
\[ \mu(( A \cdot b_3) \setminus X) \geq \frac{\mu(A)(\mu(B)-\pi(X))}{\mu(B)} \geq \mu(B) - \pi(X),  \]
where we have used the hypothesis that $\mu(B) \leq \mu(A)$. Combining this with \eqref{xs}, we get that
\[
\mu(A \cdot \{b_1,b_2,b_3\}) = \mu(X) + \mu((A \cdot b_3) \setminus X)   \geq \mu(A)+\mu(B),
\]
which concludes the proof of Lemma \ref{lem: A B when A is large} .
\end{proof}

With this in hand, we will now essentially divide our proof of Theorem \ref{thm:main} into two cases, the first being the setting when the set of popular products $A\cdot_\delta B$, for some appropriate choice of $\delta$, is somewhat large, wherein, probabilistic methods suffice to dispense the desired result. The second case is when the set of popular products is small. Here, we will first apply an inverse theorem to show that our sets $A,B$ can be mapped surjectively to large subsets of some intervals in $\TT$, whereupon, we will apply Lemma \ref{lem: A B when A is large} to deduce the claimed estimate. We now present the deduction of our main result in the first of the above two cases.

    \begin{lemma}\label{lem: popular case}
     Let $G$ be a connected, compact group, let $A,B\subseteq G$ be compact sets and let $0<\delta<\min\{\mu(A),\mu(B)\}$ be a real number. Then for every $c\geq (\mu(B)/\delta)^2$, there exist $b_1,\dots,b_c\in B$ such that 
     \[
     \mu(A \cdot \{b_1,\dots,b_c\})\geq (1-O(\exp(-c^{1/2})))\mu(A\cdot_\delta B). 
     \]
    \end{lemma}
    \begin{proof}
We choose a set with $c$ elements $B_c:=\{b_1,\dots,b_c\}$ from $B$ uniformly at random. 
Note that an element $g$ is in $A \cdot B_c$ if and only if at least one of $b_i$ in $b_1,\ldots,b_c$ satisfies that $b_i^{-1}\in g^{-1}A$. Thus, the probability that $g \in A \cdot B_c$ is
\[
1-\left(1-\frac{\mu(B\cap A^{-1}g)}{\mu(B)}\right)^c = 1-\left(1-\frac{\mathbbm{1}_{A}*\mathbbm{1}_{B}(g)}{\mu(B)}\right)^c, 
\]
and so, whenever $g\in A\cdot_{\delta} B$, then the probability that $g\in A \cdot B_c$ is at least $1-(1-\delta/ \mu(B))^c$. Since $c\geq (\mu(B)/\delta)^2$, we may deduce that 
\[
\left(1-\frac{\delta}{\mu(B)}\right)^c\leq  \exp\left(-\frac{c\delta}{\mu(B)}\right)\ll \exp(-c^{1/2}),
\]
with the first inequality above following from the fact that for every $x\in(0,1)$, one has $1-x\leq \exp(-x)$. Thus by Markov's inequality, there exists a choice of $\{b_1,b_2,\dots,b_c\}$ such that
\[
\mu( (A\cdot_{\delta} B)  \setminus (A \cdot \{b_1,b_2,\dots,b_c\}) )\ll \exp(- c^{1/2})\mu(A\cdot_{\delta} B). 
\]
This, in turn, implies that
\begin{align*}
\mu(A \cdot \{b_1,b_2,\dots,b_c\})\geq (1-O(\exp(-c^{1/2})))\mu(A\cdot_{\delta} B),
\end{align*}
as desired. 
\end{proof}

    We are now ready to prove  Theorem~\ref{thm:main}.  
    
\begin{proof}[Proof of Theorem~\ref{thm:main}]
Given $\varepsilon>0$ and let $\eta=\eta(\varepsilon)$ be as in Theorem~\ref{thm: inverse Kemperman}. Let us first consider the case when
\[
\mu(A\cdot_{\eta \mu(B)} B)\leq \mu(A)+\mu(B)+\eta \mu(B).
\]
As $\mu(A)+\mu(B)\leq 1-\beta$ with some $\beta>0$, by letting $\varepsilon$ sufficiently small we may assume the right hand side of the above inequality is smaller than $1$.  Thus by Theorem~\ref{thm: inverse Kemperman}, there is a continuous surjective group homomorphism $\chi: G\to\TT$, and two compact intervals $I,J\subseteq \TT$, so that with $\lambda$ the normalised Lebesgue measure on $\TT$, 
\[
\lambda(I)\leq (1+\varepsilon)\mu(A),\qquad \lambda(J)\leq (1+\varepsilon)\mu(B),
\]
and $A\subseteq \chi^{-1}(I)$, $B\subseteq \chi^{-1}(J)$. By choosing a sufficiently small $\varepsilon$, we may assume $\lambda(I)+\lambda(J)<1$. Then by Lemma~\ref{lem: A B when A is large}, there are $b_1,b_2,b_3\in B$ such that 
\[
\mu(A \cdot \{b_1,b_2,b_3\})\geq \mu(A)+\mu(B). 
\]

 It remains to consider the case when
 \[
\mu(A\cdot_{\eta\mu(B)} B) > \mu(A)+\mu(B)+\eta \mu(B).
\]
By Lemma~\ref{lem: popular case}, for every $c\geq \eta^{-2}$ there exist $b_1,\dots,b_c$ from $B$ such that 
\begin{align*}
\mu(A \cdot \{b_1,\dots,b_c\})&\geq (1-O(\exp(-c^{1/2})))\mu(A\cdot_{\eta\mu(B)} B)\\&>(1-O(\exp(-c^{1/2})))(\mu(A)+\mu(B)+\eta\mu(B)) \\
& \geq \mu(A)+\mu(B),
\end{align*}
where the latter inequality follows by choosing $c$ to be sufficiently large compared to $\alpha$ so as to ensure that
\[
\frac{\eta\mu(B)}{2}\geq \frac{\eta\alpha\mu(A)}{2}\gg \exp(-c^{1/2})\mu(A).
\]
This
finishes the proof of Theorem \ref{thm:main}.
    \end{proof}

We conclude this section by notation that our inverse result Theorem~\ref{thm:main inverse} follows immediately from a combination of Theorem~\ref{thm: inverse Kemperman} and Lemma~\ref{lem: popular case}. 
\begin{proof}[Proof of Theorem~\ref{thm:main inverse}]
    We will proceed by showing that the contra-positive statement holds true. Thus, given $\varepsilon,\beta>0$, suppose that the conclusion of Theorem~\ref{thm:main inverse} does not hold. Applying Theorem~\ref{thm: inverse Kemperman}, we may find $\eta'>0$ such that 
    \[
    \mu(A\cdot_{\eta'\mu(B)} B)>\mu(A)+\mu(B)+\eta'\min\{\mu(A),\mu(B)\}. 
    \]
    Now by Lemma~\ref{lem: popular case}, for every $c \geq (\eta')^{-1}$, there exist $b_1,\dots,b_c$ in $B$ such that 
    \[
    \mu(A \cdot \{b_1,\dots,b_c\})>(1-O(\exp(-c)))(\mu(A)+\mu(B)+\eta'\min\{\mu(A),\mu(B)\}). 
    \]
   As $\mu(B)>\alpha\mu(A)$, we may now set $\eta=\eta'/2$ and note that whenever $c$ is large enough, we will get 
   \[
   \mu(A \cdot \{b_1,\dots,b_c\})>\mu(A)+\mu(B)+\eta\min\{\mu(A),\mu(B)\}.
   \]
   This finishes the proof of Theorem~\ref{thm:main inverse}. 
    \end{proof}

\section{Proof of Theorem $\ref{dis}$}


In this section, we will present the proof of Theorem \ref{dis} and we begin by recording the following weaker version of Theorem \ref{dis} where we allow the parameter $c=c(d)$ to also depend on the number of translates of a one dimensional subspace that we would require to cover $A$.

\begin{lemma} \label{lin}
For every $d,r \in \mathbb{N}$ with $d \leq r$, there exists $s \in \mathbb{N}$ with $s \leq 3d r^2$ such that the following holds true. Let $A$ be a finite, non-empty subset of $\mathbb{Z}^d$ such that $\dim(A) =d$ and $A \subseteq l_1 \cup \dots \cup l_r$, where $l_1, \dots, l_r$ are parallel lines such that
\[ |A \cap l_i| \geq 2\]
for each $1 \leq i \leq r$. Then there exist $a_1, \dots, a_{s} \in A$ such that
\[ |A + \{ a_1, \dots, a_s\} | \geq (d+1) |A| - 3dr. \]
\end{lemma}

We note that Lemma \ref{lin} and Theorem \ref{mnz} combine together to deliver Theorem \ref{dis} in a straightforward manner.

\begin{proof}[Proof of Theorem \ref{dis}]
Let $A \subseteq \mathbb{R}^d$ be a finite, non-empty set with $\dim(A) =d$. Applying Theorem \ref{mnz}, we can either find $a_1,\dots, a_c \in A$ such that
\[ |A + \{a_1, \dots, a_c\}| \geq (d+ 1 + 1/16)|A|,\]
or we have that $|A| \ll_d 1$ or $A$ is contained in at most $(d+1)^2$ translates of some line $l$. We are done in the first case, and so, suppose that $|A| \ll_d 1$. In this case, we may take $\{a_1, \dots, a_c\} = A$ and apply \eqref{fl} to deduce that
\[ |A+ \{a_1, \dots, a_c\}| = |A+A| \geq (d+1) |A| - d(d+1)/2,\]
which is better than the desired bound. Thus, we now consider the final case where $A$ is contained is at most $(d+1)^2$ translates of some line $l$. Here, we may remove at most $2(d+1)^2$ many elements from $A$ to further assume that $A$ is covered by translates $l_1,\dots, l_r$ of $l$ with $r\leq (d+1)^2$, such that $|A \cap l_i| \geq 2$ for every $1 \leq i \leq r$. In this setting, we may now apply Lemma \ref{lin} to obtain the desired bound in a straightforward manner.
\end{proof}

Thus, our aim for the rest of this section is to prove Lemma \ref{lin}.

\begin{proof}[Proof of Lemma \ref{lin}]
We will prove this lemma by induction, and thus, note that when $d=1$, this follows from \eqref{b12}. We may now assume that $d \geq 2$, in which case, we denote  $H$ to be the $(d-1)$-dimensional subspace orthogonal to $l_1$ and we let $\pi : \mathbb{R}^d \to H$ be the natural projection map. We write $p_i = A \cap l_i$ and $x_i = \pi(p_i)$ for every $1 \leq i \leq r$ as well as $x^{\pi} = \pi^{-1}(x) \cap A$ for every $x \in H$. This allows us to define $X = \pi(A) = \{x_1, \dots, x_r\}$. Note that since $\dim(A) = d$, we have that $\dim(X) = d-1$. Let $\mathcal{C}$ be the convex hull of $X$ and without loss of generality, let $x_r \in X$ be an extreme point of $X$, that is, let $x_r$ be a vertex on the convex hull $\mathcal{C}$ of X. We now consider the set $X' = X \setminus \{x_r\}$ and its preimage $A' = \{ a \in A : \pi(a) \in X' \}$.
\par

Noting that $d-2 \leq \dim(X') \leq d-1$, we divide our proof into two cases, the first being when $\dim(X') = d-2$. In this case, since $\dim(X )= d-1$, we see that $x$ does not lie in the affine span of $X'$, whence, the sets 
\[ \{2x_r\},   X' + \{x_r\} \ \ \text{and} \ \ X' + X' \]
are pairwise disjoint. This, in turn, implies that the sets
\[ x_r^{\pi} + x_r^{\pi}, x_r^{\pi} + x_1^{\pi}, \dots, x_r^{\pi} + x_{r-1}^{\pi}, A' + A'\]
are pairwise disjoint. We may now apply the inductive hypothesis for $A'$ to deduce that there exist $a_1, \dots, a_{s'} \in A'$, with $s' \leq 3(d-1)(r-1)^2$, such that
\begin{equation} \label{dip}  
 |A' + \{a_1,\dots, a_{s'}\}| \geq d|A'| - 3(d-1)(r-1). \end{equation}
Let $I_1 =\{ 1\leq i \leq r: |x_r^{\pi}| \geq |x_i^{\pi}| \}$ and let $I_2 = \{1,\dots, r\}\setminus I_1$. Given $i \in I_1$, we can apply \eqref{b12} for the sumset $x_r^{\pi} + x_i^{\pi}$ to obtain elements $a_{i,1}, \dots, a_{i,3} \in x_i^{\pi}$  such that
\[ |x_r^{\pi} + \{a_{i,1}, \dots, a_{i,3}\}| \geq |x_r^{\pi}| + |x_i^{\pi}| -1. \]
Similarly, for every $i \in I_2$, we can apply \eqref{b12} for the sumset $x_r^{\pi} + x_i^{\pi}$ to obtain elements $a_{i,1}, \dots, a_{i,3} \in x_r^{\pi}$  such that
\[ |x_i^{\pi} + \{a_{i,1}, \dots, a_{i,3}\}| \geq |x_r^{\pi}| + |x_i^{\pi}| -1. \]
Putting these two together along with \eqref{dip}, we obtain a set 
\[ B = \{a_{1,1}, \dots, a_{r,3}, a_1, \dots, a_{s'}\} \subseteq A \]
with $|B| \leq s' + 3r$ such that
\begin{align*}
|A+ B| & \geq |A' + \{a_1, \dots, a_{s'}\}| + \sum_{i \in I_1} |x_{r}^{\pi} + \{a_{i,1}, \dots, a_{i,3}\} | + \sum_{i \in I_2} |x_{i}^{\pi} + \{a_{i,1}, \dots, a_{i,3}\} |  \\
& \geq d|A'| - 3(d-1)(r-1) + \sum_{1 \leq i \leq r} (|x_r^{\pi}| + |x_i^{\pi}| - 1) \\
& \geq (d+1) |A'| + (r+1)|x_r^{\pi}| - r  - 3(d-1)(r-1) \\
& \geq (d+1) |A| - 3dr.
\end{align*}
Moreover, by the inductive hypothesis, we have 
\[ |B| \leq s' + 3r \leq 3(d-1)(r-1)^2 + 3r \leq 3dr^2 , \]
and consequently, we obtain the required conclusion in this case.
\par

We now consider the case when $\dim(X')= d-1$. In this case, we analyse the convex hull $\mathcal{C}'$ of $X'$. Here, since $\dim(X') = d-1$ and since $x_r$ is an extreme point of $\mathcal{C}$, we find elements $y_1, \dots, y_{d-1} \in X'$ such that 
\[ \{x_r, (x_r + y_1)/2, \dots, (x_r + y_{d-1})/2 \} \cap \mathcal{C}' = \varnothing .\]
This implies that the sumsets
\[ x_r^{\pi} + x_r^{\pi}, x_r^{\pi} + y_1^{\pi} , \dots, x_r^{\pi} + y_{d-1}^{\pi}, A' + A' \]
are pairwise disjoint. We see that $\dim(A') = d$ because $\dim(X') = d-1$, and that $A'$ is contained in $r-1$ parallel lines. Thus, we may apply the inductive hypothesis for $A'$ to obtain $a_1, \dots, a_{s'} \in A'$, with $s' \leq 3d(r-1)^2$, such that
\[  |A' + \{a_1, \dots, a_{s'} \}| \geq (d+1)|A'| - 3d(r-1). \]  
Moreover, for every $1 \leq i \leq d-1$, we fix an element $b_i \in y_i^{\pi}$. Furthermore, we apply \eqref{b12} for the set $x_r^{\pi} + x_r^{\pi}$ to obtain elements $b_{d}, b_{d+1}, b_{d+2} \in x_r^{\pi}$ satisfying
\[ |x_r^{\pi} + \{b_{d}, b_{d+1}, b_{d+2}\}| \geq 2|x_r^{\pi}| -1. \]
As in the previous case, we see that the set $B = \{a_1, \dots, a_{s'},b_1, \dots, b_{d+2}\} \subseteq A$ satisfies
\begin{align*}
|A+ B| & \geq    |A' + \{a_1, \dots, a_s\}| + |x_{r}^{\pi} + \{b_{d}, b_{d+1}, b_{d+2}\} |  + \sum_{1\leq i \leq d-1} |x_{r}^{\pi} + \{b_i\}|\\
& \geq  (d+1)|A'| - 3d(r-1) + 2|x_r^{\pi}| - 1 + (d-1)|x_r^{\pi}|     \\
& \geq (d+1)|A| - 3dr.
\end{align*}
Moreover, by the inductive hypothesis, we have that
\[ |B| \leq s' + d+2 \leq  3d(r-1)^2 + 3d \leq 3dr^2,\]
and so, we finish our proof of Lemma \ref{lin}.
\end{proof}

\section{Proof of Theorem $\ref{mnz}$}

Our goal in this section is to present the proof of Theorem $\ref{mnz}$. We begin by noting the following asymmetric generalisation of Freiman's lemma which was given by Ruzsa \cite{Ru1994}. Thus, whenever $A, B \subseteq \mathbb{R}^d$ are finite sets with $\dim(A) = d$ and $|A| \geq |B|$, then 
\[ |A+B| \geq |A| + d|B| - d(d+1)/2. \]
Our first objective in this section is to show that in a slightly more specific version of this setting, whenever $A+B$ is close to the above lower bound, then both $A$ and $B$ can be efficiently covered by translates of a one dimensional subspace, which, in turn, implies that both $A$ and $B$ have dense subsets lying on translates of the same line.

\begin{lemma} \label{cov}
    Let $A,B \subseteq \RR^d$ be finite sets with $\dim(A) = d$ and $|A| = |B|$ and
    \[ |A+B| \leq |A| + (d+1/7)|B| - O_d(1). \]
    Then either $|A| \ll_d 1$ or there exists some line $l$ and some $x, y \in \mathbb{R}^d$ such that
    \[ |A \cap (x+l)| \geq |A|/d \ \ \text{and} \ \ |B \cap (y+l)| \geq |B|/d.  \]
    \end{lemma}

    \begin{proof}
This is true trivially when $A, B \subseteq \RR$, whence, we may assume that $d \geq 2$. Applying Lemma \ref{prin} with $k=2$, we may deduce that 
 \[ |A+A| \leq (d+ 8/7)^2 |A| . \]
 We may now apply Lemma \ref{m1} to obtain parallel lines $l_1, \dots, l_r$, with $r \ll_d |A|^{1- \sigma}$ for some $\sigma \in (0,1/2]$ that only depends on $d$, such that $A \subseteq l_1 \cup \dots \cup l_r$ and
 \[ |l_1 \cap A| \gg_d |A|^{\sigma} . \]
 We now cover $B$ with lines that are parallel to $l_1$, and so, let $q \in \mathbb{N}$ be the minimal natural number such that $B \subseteq m_1\cup \dots \cup m_q$, where $m_1, \dots, m_q$ are lines parallel to $l_1$. Note that $q \ll_d |A|^{1- \sigma}$, since
 \[ (d+8/7)|A| \geq |A+B| \geq \sum_{i=1}^{q} |l_1 + m_i| \geq q |l_1| \gg_{d} q |A|^{\sigma} . \]
 This allows us to apply Lemma \ref{m2}, whence, we have
 \[ |A+B| \geq |A| + \Big( d + 1 - \frac{1}{r-d + 2} - \frac{1}{q - c+ 2}\Big)|B| - (d-1)(r+q),\]
where $c = \min \{d, \dim(B)\}$. Here, since $r,q \ll_{d} |A|^{1- \sigma}$, we may combine this with the hypothesis of Lemma \ref{cov} to see that
\[ \Big( \ \frac{6}{7}  - \frac{1}{r-d + 2} - \frac{1}{q - c+ 2} \ \Big)|B| \ll_{d} r+q \ll_{d} |A|^{1- \sigma}.  \]
Note that if $r > d$ or $q > c$, then 
\[ 1/(r- d+2) + 1/(q-c+2) \leq 1/2 + 1/3 = 5/6, \]
which combines with the above inequality to gives us $|A| = |B| \ll_d 1$. On the other hand, when $r = d$ and $q = c$, then we must have some $1 \leq i \leq r$ and $1 \leq j \leq q$ such that
\[ |A \cap l_i| \geq |A|/r = |A|/d \ \ \text{and} \ \ |B \cap m_j| \geq |B|/q = |B|/c \geq |B|/d. \]
Thus, we conclude our proof of Lemma \ref{cov}.
    \end{proof}

With this in hand, we will now proceed with our proof of Theorem \ref{mnz}.

\begin{proof}[Proof of Theorem \ref{mnz}]
Our aim is to show that given a finite, non-empty set $A \subseteq \mathbb{R}^d$ such that $\dim(A) =d$, we either have $|A| \ll_d 1$ or $A$ is covered by at most $(d+1)^2$ translates of some line or we can find elements $a_1,\dots,a_c \in A$ such that
\[ |A+ \{a_1, \dots, a_c\}| \geq (d+1 + 1/16)|A|. \]
Setting $K = (d+1 + 1/10)$ and $\varepsilon = 100^{-d^2}$, we apply Lemma $\ref{btb}$ to obtain $c = c(d) >0$ and a set $A^{*} \subseteq A$ such that $|A^*| \geq (1 - 100^{-d^2})|A|$ and such that if we select $a_1,\dots, a_c \in A^*$ uniformly at random from $A^*$, then
\[ \mathbb{E}_{a_1,\dots,a_c \in A^*}  |A^* + \{a_1,\dots, a_c\}| \geq \min \{(1 - 100^{-d^2})|A^{*} + A^{*}|, (d + 1 + 1/10)|A^{*}|  \} .  \]
Here, if 
\begin{equation} \label{stp}
(1 - 100^{-d^2})|A^{*} + A^{*}| \geq  (d + 1 + 1/10)|A^*|,  \end{equation}
then we have that
\begin{align*}
\mathbb{E}_{a_1,\dots,a_c \in A^*}  |A^* + \{a_1,\dots, a_c\}| & \geq (1 - 100^{-d^2}) (d + 1 + 1/10)|A|  \\
& \geq (d + 1 + 1/12) |A|, 
\end{align*}
whence, we are done. Thus, we may assume that \eqref{stp} does not hold true, in which case, we obtain elements $a_1, \dots, a_c \in A^*$ such that
\begin{equation} \label{utm}
    |A^* + \{a_1,\dots, a_c\}| \geq  (1 - 100^{-d^2})|A^{*} + A^{*}| . 
\end{equation} 
\par

We now suppose that $\dim(A^*) = m$ for some $1 \leq m \leq d$. We first consider the subcase when 
\begin{equation} \label{as*}
|A^* + A^*| \geq (m + 1 + 1/10)|A^*| - m(m+1)/2. 
\end{equation}
Here, since $\dim (A) = d$, we see that there exist elements $b_{m+1},\dots, b_{d} \in A \setminus A^*$ which are linearly independent and are not contained in the affine span of $A^*$. Thus, we may now set $A' = \{a_1, \dots, a_c\} \cup \{b_{m+1},\dots, b_d\}$ to see that
\begin{align*}
|A+ A'| 
& \geq |A^* + A'| = |A^* + \{a_1, \dots, a_c\}| + \sum_{i=m+1}^d |A^* +  \{b_i\}|   \\
& \geq (1 - 100^{-d^2})|A^{*} + A^{*}|  + (d-m)|A^*| \\
& \geq (1- 100^{-d^2})(m+1 + 1/10)|A^*| + (d-m)|A^*| - m(m+1)/2  \\
& \geq (d + 1 + 1/15) |A^*| - m(m+1)/2 \\
&\geq (d+1+2/31)|A| - d(d+1),
\end{align*}
where the second and third inequalities utilise \eqref{utm} and \eqref{as*} respectively. Thus we either have that
\[ |A+A'| \geq (d + 1 + 1/16)|A| \]
or $|A| \ll_d 1$, and so, we are done in this subcase.
\par

 Noting the above, we may now assume that \eqref{as*} does not hold true. After applying suitable affine transformations, we may employ Lemma \ref{cov} to show that either $|A| \ll_{m} 1$ or there exists some line $l$ in $\RR^d$ such that
\begin{equation} \label{dense}
 |A \cap l| \geq |A^* \cap l| \geq |A^*|/m \geq |A|(1- 100^{-d^2})/d .  
 \end{equation}
Since $m \leq d$ we are done in the setting when $|A| \ll_m 1$, and so, we may assume that some $l$ exists such that \eqref{dense} holds. We now cover $A$ with translates of the line $l$, and so, let $l_1, \dots, l_r$ be lines parallel to $l$ such that 
\[ A \subseteq l_1 \cup \dots \cup l_r ,\]
and $r$ is minimal. We may assume that $r \geq (d+1)^2+1$, since otherwise, we are done. Moreover, writing $p_i = A \cap l_i$ for every $1 \leq i \leq r$, we may assume that 
\[ |p_1| \geq \dots \geq |p_r|,\]
and so, 
\[ |p_1| \geq |A \cap l| \geq |A|(1- 100^{-d^2})/d . \]
Writing $r_0 = (d+1)^2 +1$, observe that the sumsets 
\[ p_1 + p_1, \dots, p_1 + p_{r_0}\]
are pairwise disjoint, and so, we may apply \eqref{b12} for each such sumset to procure elements $a_{i,1}, a_{i,2}, a_{i,3} \in p_i$, for every $1\leq i \leq r_0$, such that
\begin{align*}
    |A + \{ a_{1,1}, \dots, a_{r_0,3}\}|  
    & \geq \sum_{i=1}^{r_0} |p_1 + \{ a_{i,1}, \dots, a_{i,3}\}|  \\
    & \geq \sum_{i=1}^{r_0} (|p_1| + |p_i| - 1) \geq r_0 |p_1| \\
    & > (d+1)^2 |A|(1 - 100^{-d^2})/d \\
    & \geq (d+2)|A|,
\end{align*}  
This implies the desired bound, and consequently, we conclude the proof of Theorem \ref{mnz}.
\end{proof}

\bibliographystyle{amsplain}
\bibliography{reference}
\end{document}